\documentclass[12pt,reqno]{amsart}
\usepackage{amsmath,amsopn,amssymb,amsthm,multicol}
\usepackage{hyperref}
\usepackage{graphicx}

\newcommand{\fr}{\mathfrak}

 \newtheorem{lemma} {Lemma} [section]
\newtheorem{theorem}[lemma]{Theorem} 
\newtheorem{remark}[lemma] {Remark} 
\newtheorem{prop} [lemma]{Proposition}  
 
\newtheorem{corol}[lemma] {Corollary} 
 
\newtheorem{problem}[lemma] {Problem}

\begin{document}

\title{Geodesic orbit metrics in compact homogeneous manifolds with equivalent isotropy submodules}

\author{Nikolaos Panagiotis Souris}
\address{University of Patras, Department of Mathematics, GR-26500 Patras, Greece}
\email{nsouris@upatras.gr}

\begin{abstract}
A geodesic orbit manifold (GO manifold) is a Riemannian manifold $(M,g)$ with the property that any geodesic in $M$ is an orbit of a one-parameter subgroup of a group $G$ of isometries of $(M,g)$.  The metric $g$ is then called a $G$-GO metric in $M$.  For an arbitrary compact homogeneous manifold $M=G/H$, we simplify the general problem of determining the $G$-GO metrics in $M$.   In particular, if the isotropy representation of $H$ induces equivalent irreducible submodules in the tangent space of $M$, we obtain algebraic conditions, under which, any $G$-GO metric in $M$ admits a reduced form.  As an application we determine the $U(n)$-GO metrics in the complex Stiefel manifolds $V_k\mathbb C^n$.\end{abstract}

\title[GO metrics in compact manifolds with equivalent isotropy submodules]{Geodesic orbit metrics in compact homogeneous manifolds with equivalent isotropy submodules}

\maketitle
\medskip   
\section{Introduction}\label{intro}

The class of GO manifolds is a proper subclass of the class of \emph{D'Atri manifolds}, i.e, the local geodesic symmetries are volume preserving (\cite{Ko-Va-1}), and its most notable proper subclasses include those of \emph{symmetric spaces}, \emph{naturally reductive spaces} (\cite{Kob-No}), \emph{weakly symmetric manifolds} (\cite{Ber-Ko-Va},\cite{Wo}), as well as the classes of \emph{Clifford-Wolf homogeneous manifolds} (\cite{Be-Ni-2}) and \emph{$\delta$-homogeneous manifolds} (\cite{Be-Ni-1}).  The classification of GO manifolds was initially considered in 1991 by Kowalski and Vanhecke, who classified the GO manifolds up to dimension 6 (\cite{Ko-Va-2}).  In \cite{Go} the classification of GO manifolds is partially reduced to the classification of GO nilmanifolds, compact GO manifolds and GO manifolds with a non-compact semisimple group of isometries.  Several classes of GO manifolds have also been characterized (\cite{Al-Ar}, \cite{Al-Ni}, \cite{Ca-Ma}, \cite{Ni-1}, \cite{Ta} to name a few), however their complete classification remains an open problem.\\
Our study restricts to compact GO manifolds, and the purpose of our work is to simplify the general problem of determining the $G$-GO metrics $g$ in an arbitrary compact homogeneous manifold $M$ diffeomorphic to $G/H$.  The main obstacle against this goal is the existence of ''non-diagonal'' elements of $g$ resulting from \emph{isotypical summands} (sums of equivalent submodules) of the \emph{isotropy representation} of $H$ (section \ref{inme}).  For that reason we obtain algebraic conditions in the tangent space of $M$ which allow us to simplify the form of the candidate $G$-GO metrics $g$ in $M$ (section \ref{s3}).  In particular, assume that $p$ is the origin in $M$, $\chi:H\rightarrow \operatorname{Aut}(T_pM)$ is the isotropy representation and $A:T_pM\rightarrow T_pM$ is the corresponding \emph{metric endomorphism} of a $G$-GO metric $g$, expressed in terms of a chosen basis.  The aforementioned conditions allow us to exploit the Lie bracket relations between the irreducible submodules of $\chi$, primarly in order to eliminate the non-diagonal elements of the restriction of $A$ on certain isotypical summands (Propositions \ref{en}, \ref{class1} and \ref{class}), and secondly, to reduce the number of distinct eigenvalues of $A$ (Proposition \ref{triple}).  Consequently, the problem of obtaining the $G$-GO metrics in $M$, is largely reduced to the process of determining the classes of equivalent, irreducible isotropy submodules and the Lie algebraic relations between those submodules.  Our approach and our results regarding the investigation of GO metrics in compact homogeneous manifolds are presented in sections \ref{mot} and \ref{s3} respectively.\\
Finally, in section \ref{prooof} we implement our results in order to obtain the $U(n)$-GO metrics in the complex Stiefel manifolds $V_k\mathbb C^n$.  In particular we prove the following.

\begin{theorem}\label{mainuu}Let $V_k\mathbb C^n$ be a complex Stiefel manifold.  Then $V_k\mathbb C^n$ admits exactly one (up to scalar) family of $U(n)$-GO metrics $A_{t}$, $t>0$.  The metrics $A_t$ are smooth deformations of the normal metric $A_1$, along the center of the group $N_G(H)/H$, where $G=U(n)$, $H=U(n-k)$ and $N_G(H)$ is the normalizer of $H$ in $G$.
\end{theorem}

\noindent
{\bf Acknowledgements.}
The work was supported by Grant $\#E.037$ from the Research Committee of the University of Patras (Programme K. Karatheodori).
The author wishes to acknowledge Professors A. Arvanitoyeorgos and Y. G. Nikonorov for their useful comments on the manuscript. 

\section{Investigation of GO metrics in compact homogeneous manifolds}\label{mot}

\subsection{Problem statement}\label{ps}

A \emph{geodesic orbit manifold} (or a \emph{GO manifold}) is a Riemannian manifold $(M,g)$ such that any geodesic in $M$ is an orbit of an one-parameter subgroup of a Lie group $G$ of isometries of $(M,g)$.  To emphasize the group $G$, the manifold $(M,g)$ is also called a $G$-GO manifold.  In turn, the metric $g$ is called a \emph{$G$-GO metric} in $M$.  The definition of a GO manifold can also be extended to pseudo-Riemannian manifolds (\cite{Du}).  If $(M,g)$ is a connected $G$-GO manifold then $M$ is \emph{homogeneous} and therefore diffeomorphic to the space $G/H$ where $H$ is the isotropy subgroup of $p\in M$.  The corresponding $G$-GO metric $g$ is a \emph{$G$-invariant metric} in the space $G/H$, i.e. it is invariant by the left translations $\tau_q:G/H\rightarrow G/H$, $q\in G$.  The Riemannian space $(G/H,g)$ is then called a \emph{GO space}.\\
 Thus, in order to determine all GO metrics in a given homogeneous manifold $M$, one needs first to determine the $G$-GO metrics in $M$ for any Lie group $G$ acting smoothly and transitively on $M$.  For any such choice $G$, the above process then restricts to obtaining the $G$-invariant metrics $g$ such that the corresponding Riemannian space $(G/H,g)$ is a GO-space. \\ 
Our study concerns the class of compact homogeneous manifolds.  We consider the following general problem.

\begin{problem}\label{pro1}Let $M$ be a compact homogeneous manifold and let $G$ be a Lie group acting smoothly and transitively on $M$ so that $M$ is diffeomorphic to $G/H$.  Determine the $G$-GO metrics $g$ in $M$.\end{problem}

Let $p\in M$ be the origin and let $\fr{g},\fr{h}$ be the Lie algebras of $G,H$ respectively.  Since $M$ is compact, then $G$ is also compact, hence there exists an \emph{$\operatorname{Ad}$-invariant inner product} $B$ in $\fr{g}$.  We consider the $B$-orthogonal \emph{reductive decomposition} 

\begin{equation}\label{decc}\fr{g}=\fr{h}\oplus \fr{m},\end{equation} 

 where $\fr{m}$ is isomorphic to the tangent space $T_p(G/H)$, and $\operatorname{Ad}(H)\fr{m}\subset \fr{m}$.  Then any $G$-invariant metric $g$ in $G/H$ is in 1-1 correspondence with an \emph{$\operatorname{Ad}(H)$-equivariant, symmetric and positive definite} endomorphism $A:\fr{m}\rightarrow \fr{m}$, called the \emph{metric endomorphism} of $g$, such that

\begin{equation}\label{atp}g_p(X,Y)=B(AX,Y), \quad X,Y \in \fr{m}.\end{equation}

In order to find those metric endomorphisms $A$ corresponding to a $G$-GO metric, we use the following condition which is based on a proposition in \cite{Al-Ar}.
	
\begin{prop}\label{G}Let $M=G/H$ be a compact homogeneous manifold, with the decomposition $\fr{g}=\fr{h}\oplus \fr{m}$.  Then $(M,g)$ is a $G$-GO manifold, (i.e. $(G/H,g)$ is a GO-space) if and only if for any vector $X\in \fr{m}$, there exists a vector $a=a_X\in \fr{h}$ such that 

\begin{equation*}\label{cond}[a+X,AX]=0,\end{equation*}

where $A:\fr{m}\rightarrow \fr{m}$ is the metric endomorphism of $g$.
\end{prop}
	
\begin{proof}	
	
According to Proposition 1. in \cite{Al-Ar}, the metric $g$ is a $G$-GO metric if and only if for any vector $X\in \fr{m}$, there exists a vector $a=a_X\in \fr{h}$ such that $[a+X,AX]\in \fr{h}$ .  Thus it suffices to prove that $[a+X,AX]\in \fr{m}$.\\
From the $\operatorname{Ad}(H)$-invariance of $\fr{m}$ and the definition of $A$ we obtain that $[a,AX]\in \fr{m}$.  Therefore, it remains to prove that $[X,AX]\in \fr{m}$.  Indeed, choose any vector $\widetilde{a}\in \fr{h}$.  By using the $\operatorname{ad}(\fr{h})$-equivariance and the $B$-symmetry of $A$ as well as the $\operatorname{ad}$-skew symmetry of $B$ (resulting from the $\operatorname{Ad}$-invariance of $B$), we obtain that 

\begin{equation*}B([X,AX],\widetilde{a})=-B(X,[\widetilde{a},AX])=-B(X,A[\widetilde{a},X])=-B(AX,[\widetilde{a},X])=-B([X,AX],\widetilde{a}),\end{equation*}

\noindent which implies that $B([X,AX],\widetilde{a})=0$.  Therefore $[X,AX]\in \fr{m}$, hence Proposition \ref{G} follows.

\end{proof}

 Under the above notation, Problem \ref{pro1} reduces to the following.
\begin{problem}\label{pro2}Determine the metric endomorphisms $A:\fr{m}\rightarrow \fr{m}$ with the property that for any vector $X\in \fr{m}$, there exists a vector $a=a_X\in \fr{h}$ such that 
\begin{equation}\label{obv}[a+X,AX]=0.\end{equation}
\end{problem}

Let $N_G(H_o)$ be the normalizer in $G$ of the unit component $H_o$ of $H$, and let $\fr{n}_{\fr{g}}(\fr{h})$ be the corresponding Lie algebra.  We state an important property for a GO-metric.

\begin{prop}\label{nik} (\cite{Ni-2})  Let $(G/H,\langle \ ,\ \rangle)$ be a GO space.  Then the product $\langle \ ,\ \rangle$ is $\operatorname{Ad}(N_G(H_o))$-invariant. 
\end{prop}

As a result we obtain the following.

\begin{corol}\label{poritsis} Let $(G/H,\langle \ ,\ \rangle)$ be a GO space.  The metric endomorphism $A:\fr{m}\rightarrow \fr{m}$ is $\operatorname{Ad}(N_G(H_o))$-equivariant.
\end{corol}

\subsection{The general form of a metric endomorphism}\label{inme}

To obtain the general form of a metric endomorphism $A:\fr{m}\rightarrow \fr{m}$, one has to compute the \emph{isotropy representation} $\chi:H\rightarrow \operatorname{Aut}(\fr{m})$, given by $\chi(h)=(d\tau_h)_p:\fr{m}\rightarrow \fr{m}$, $h\in H$.  The representation $\chi$ is equivalent to the adjoint representation of $H$ in $\fr{m}$, and since $H$ is compact (as a closed subgroup of $G$), it induces a $B$-orthogonal splitting

\begin{equation}\label{is}\fr{m}=\fr{m}_1\oplus \cdots \oplus \fr{m}_s\end{equation}

\noindent of $\fr{m}$ into $\operatorname{Ad}(H)$-invariant and $Ad(H)$-irreducible submodules $\fr{m}_i$, $i=1,\dots,s$.  A submodule $\fr{m}_i$ is \emph{equivalent} to a submodule $\fr{m}_j$ if there exists a non-zero \emph{$\operatorname{Ad}(H)$-equivariant map} $\phi:\fr{m}_i\rightarrow \fr{m}_j$, that is, a linear map such that $\phi\circ \operatorname{Ad}(h)=\operatorname{Ad}(h)\circ \phi$, $h\in H$.  Moreover, since $\fr{m}_i,\fr{m}_j$ are irreducible then any non-zero $\operatorname{Ad}(H)$-equivariant map $\phi:\fr{m}_i\rightarrow \fr{m}_j$ is an isomorphism.  Thus, the representation $\chi$ induces an equivalence relation on the set of its irreducible submodules.  For an irreducible submodule $\fr{m}_i\subset \fr{m}$, we denote by $C_{\fr{m}_i}$ be the corresponding \emph{equivalence class}, that is

\begin{equation*}\label{eqcl}C_{\fr{m}_i}=\left\{ {\fr{m}_j \quad \makebox{irreducible submodule of $\chi$}:\fr{m}_j \quad \makebox{is equivalent to}\quad \fr{m}_i} \right\}.\end{equation*} 

The space 

\begin{equation*}S_{\fr{m}_i}=\bigoplus_{\fr{m}_j\in C_{\fr{m}_i}}{\fr{m}_j},\end{equation*}

\noindent is called the corresponding \emph{isotypical summand} of the class $C_{\fr{m}_i}$.  By regroupping the submodules of the decomposition (\ref{is}) into isotypical summands $S_k$, $k=0,1,\dots,N$, we obtain a $B$-orthogonal decomposition

\begin{equation}\label{is1}\fr{m}=S_0\oplus \cdots \oplus S_N, \quad N<s,\end{equation}
of $\fr{m}$.  Here $S_0$ denotes the isotypical summand of $\fr{m}$ generated by the trivial submodules of $\chi$, i.e.,

\begin{equation}\label{S0}S_0=\{X\in \fr{m}:\operatorname{Ad}(h)X=X,\quad h\in H\}.\end{equation}
Moreover, $S_0$ coincides with the Lie algebra of the group $N_G(H)/H$, where $N_G(H)$ is the normalizer of $H$ in $G$.\\
Since any metric endomorphism $A:\fr{m}\rightarrow \fr{m}$ is $\operatorname{Ad}(H)$-equivariant, symmetric, and positive definite, we obtain the following.

\begin{prop}\label{summar}Any metric endomorphism $A:\fr{m}\rightarrow \fr{m}$ decomposes as $A=\left.A\right|_{S_0}+\left.A\right|_{S_1}+\cdots +\left.A\right|_{S_N}$, where each map $\left.A\right|_{S_k}:S_k\rightarrow S_k$, $k=0,1,\dots, N$ is an $\operatorname{Ad}(H)$-equivariant, symmetric and positive definite endomorphism.\end{prop}

We fix a $B$-orthogonal basis $\mathcal{B}$ in $\fr{m}$ adapted to the decomposition (\ref{is1}).  Let $A^{\mathcal{B}}$ and $\left.A^{\mathcal{B}}\right|_{S_k}$, $k=0,\dots,N$, be the matrix representations of $A$ and $\left.A\right|_{S_k}$ respectively, in terms of $\mathcal{B}$.  Then

\begin{equation}\label{gm}A^{\mathcal{B}}=\left(\begin{array}{cccc}
\left.A^{\mathcal{B}}\right|_{S_0} &  0      &  \dots     &0    \\
   0 & \left.A^{\mathcal{B}}\right|_{S_1}    &          \\
   \vdots &  \dots      & \ddots &          \\
   0&  \dots       &  \dots      &\left.A^{\mathcal{B}}\right|_{S_N} \\
 \end{array}\right).\end{equation}

 In turn, for any isotypical summand $S_k=\fr{m}_{k_1}\oplus \cdots \oplus \fr{m}_{k_r}$, the block matrix $\left.A^{\mathcal{B}}\right|_{S_k}$ has the form 

\begin{equation}\label{matrixform}\left.A^{\mathcal{B}}\right|_{S_k}=\left(\begin{array}{cccc}
\mu_1\left.\operatorname{Id}\right|_{\fr{m}_{k_1}} &  A_{21}      &  \dots     &A_{r1}    \\
   A_{12} & \mu_2\left.\operatorname{Id}\right|_{\fr{m}_{k_2}}    &          \\
   \vdots &  \dots      & \ddots &          \\
   A_{1r} &  \dots       &  \dots      &\mu_r\left.\operatorname{Id}\right|_{\fr{m}_{k_r}} \\
 \end{array}\right), \quad \mu_1,\dots,\mu_r>0, \end{equation}

\noindent where each block matrix $A_{lm}$, $1\leq l\neq m\leq r$, corresponds to an $\operatorname{Ad}(H)$-equivariant map $\phi:\fr{m}_{k_l}\rightarrow \fr{m}_{k_m}$.  Moreover, the symmetry of $A$ implies that $A_{ml}=A_{lm}^T$, $1\leq l\neq m\leq r$.  Consequently, for any vector $X_l\in \fr{m}_{k_l}\subset S_k$, $l=1,\dots,r$, it is 

\begin{equation}\label{mfe}A^{\mathcal{B}}X_l=\left.A^{\mathcal{B}}\right|_{S_k}X_l=\mu_lX_l+\sum_{l\neq m=1}^r{A_{lm}X_l}, \quad A_{lm}:\fr{m}_{k_l}\rightarrow \fr{m}_{k_m}.\end{equation}

In particular, if $S_k$ consists of only one submodule then $\left.A^{\mathcal{B}}\right|_{S_k}$ is scalar.\\
For the rest of this paper, any $G$-invariant metric in $G/H$ will be denoted by the corresponding metric endomorphism $A$.  Moreover, if a basis $\mathcal{B}$ of $\fr{m}$ is fixed, then we will make no distinction between the endomorphism $A$ and the matrix $A^{\mathcal{B}}$.\\

A $G$-invariant metric in the manifold $M$ is called \emph{normal with respect to the decomposition \ref{decc}} if the corresponding endomorphism $A:\fr{m}\rightarrow \fr{m}$ is a scalar multiple of the identity endomorphism.  Any normal metric is an obvious solution of Problem \ref{pro2} if we choose $a=0$ in equation (\ref{obv}).  However, obtaining the non-normal solutions of Problem \ref{pro2}, if there exist any, may get particularly complex in large dimensions of $G/H$ due to the existence of non-diagonal elements for $A^{\mathcal B}$, as equation (\ref{matrixform}) shows.  Of course, for any metric endomorphism $A:\fr{m}\rightarrow \fr{m}$, one can always find a basis $\mathcal{B}_A$ of $\fr{m}$ such that $A$ is diagonal in terms of $\mathcal{B}_A$.  However, the choice of $\mathcal{B}_A$ depends on the choice of $A$.  Thus in order to investigate all metrics $A$ satisfying equation (\ref{obv}), it is preferable to fix a basis $\mathcal{B}$ and in turn obtain a simpler form for the candidate  $G$-GO metrics $A$, expressed in terms of $\mathcal{B}$.\\
 To this end, in the next section we find algebraic conditions, under which the forms (\ref{gm}) and (\ref{matrixform}) become simpler when $A$ is any $G$-GO metric.  In particular, let $A:\fr{m}\rightarrow \fr{m}$ be any $G$-GO metric in $M$.  Based on a result in \cite{Al-Ni}, we show that it is possible to identify the eigenvalues $A$ on certain $\operatorname{Ad}(H)$-invariant subspaces of $\fr{m}$ (Proposition \ref{triple}).  In Proposition \ref{en} we simplify the matrix $\left.A^{\mathcal{B}}\right|_{S_0}$, where $S_0$ is given by (\ref{S0}).  In Proposition \ref{class1} we show that it is possible to use certain values of the Lie bracket $[S_k^{\bot},S_k]$, where $S_k$ is an isotypical summand, in order to eliminate the non-diagonal elements of $\left.A^{\mathcal{B}}\right|_{S_k}$.  Finally, in Proposition \ref{class} we obtain orthogonality conditions for the $\operatorname{Ad}(H)$-equivariant isomorphisms between the submodules in $S_k$, under which, $\left.A^{\mathcal{B}}\right|_{S_k}$ is scalar.  The above results are applied in section \ref{prooof} in order to obtain the $U(n)$-GO metrics for the complex Stiefel manifolds (Theorem \ref{mainuu}).

\section{Simplification of GO metrics in compact homogeneous manifolds}\label{s3}

\subsection{Reduction of the number of distinct eigenvalues of a GO metric}\label{redux}

The following propositions permit us to reduce the number of distinct eigenvalues of any $G$-GO metric $A$, on certain pairs or triples of pairwise orthogonal $\operatorname{Ad}(H)$-invariant subspaces $\fr{m}_1,\fr{m}_2,\fr{m}_3\subset \fr{m}$.  If $A:\fr{m}\rightarrow \fr{m}$ is a metric endomorphism and $\lambda$ is an eigenvalue of $A$, we denote by $\fr{m}_{\lambda}$ the corresponding eigenspace in $\fr{m}$.  Moreover, if $\fr{m}_1,\fr{m}_2,\fr{m}_3$ are subspaces of $\fr{m}$, we denote by $[\fr{m}_1,\fr{m}_2]_{\fr{m}_3}$ the orthogonal projection of $[\fr{m}_1,\fr{m}_2]$ on $\fr{m}_3$.  The following is an important property of any $G$-GO metric $A$, proven by Alekseevsky and Nikonorov in \cite{Al-Ni}.
 
\begin{prop}\label{ni}(\cite{Al-Ni})Let $M=G/H$ be a compact homogeneous manifold with the reductive decomposition $\fr{g}=\fr{h}\oplus \fr{m}$ with respect to an $\operatorname{Ad}$-invariant inner product $B$ in $\fr{g}$.  Assume that $A$ is a $G$-GO metric in $M$ and let $\fr{m}=\fr{m}_{\lambda_1}\oplus \cdots \oplus \fr{m}_{\lambda_r}$ be the $A$-eigenspace decomposition of $\fr{m}$ such that $A|_{\fr{m}_{\lambda_i}}=\lambda_i\operatorname{Id}|_{\fr{m}_{\lambda_i}}$, $i=1,\dots, r$.  Then for any $\operatorname{Ad}(H)$-invariant subspaces $\fr{m}_i\subset \fr{m}_{\lambda_i}$, $\fr{m}_j\subset \fr{m}_{\lambda_j}$, $i\neq j$, we have 

\begin{equation}\label{e1}[\fr{m}_i,\fr{m}_j]\subset \fr{m}_i\oplus \fr{m}_j.\end{equation}

Moreover, if $\fr{m},\widetilde{\fr{m}}$ are $A$-orthogonal $\operatorname{Ad}(H)$-invariant subspaces of $\fr{m}_{\lambda_i}$, then

\begin{equation}\label{e2}[\fr{m},\widetilde{\fr{m}}]\subset \fr{h}\oplus \fr{m}_{\lambda_i} .\end{equation}
\end{prop}

From the above proposition, we obtain the following eigenvalue reduction criteria.

\begin{prop}\label{triple}Let $M=G/H$ be a compact homogeneous manifold with the reductive decomposition $\fr{g}=\fr{h}\oplus \fr{m}$ with respect to an $\operatorname{Ad}$-invariant inner product $B$ in $\fr{g}$.  Assume that $A$ is a $G$-GO metric in $M$.\\

\textbf{1.} If $\lambda_1,\lambda_2$ are eigenvalues of $A$ such that there exist $\operatorname{Ad}(H)$-invariant, pairwise $B$-orthogonal subspaces $\fr{m}_1,\fr{m}_2$ of $\fr{m}$ with

\begin{eqnarray}\label{tr1}&&\fr{m}_i\subset \fr{m}_{\lambda_i},\quad i=1,2 \quad \makebox{and} \\
&&\label{tr2}[\fr{m}_1,\fr{m}_2]_{(\fr{m}_1\oplus \fr{m}_2)^\bot}\neq \left\{ {0} \right\},\end{eqnarray}

then $\lambda_1=\lambda_2$.\\

\textbf{2.} If $\lambda_1,\lambda_2,\lambda_3$ are eigenvalues of $A$ such that there exist $\operatorname{Ad}(H)$-invariant, pairwise $B$-orthogonal subspaces $\fr{m}_1,\fr{m}_2,\fr{m}_3$ of $\fr{m}$ with 

\begin{eqnarray}\label{tr3}&& \fr{m}_i\subset \fr{m}_{\lambda_i},\quad i=1,2,3 \quad \makebox{and} \\
&& \label{tr4}[\fr{m}_1,\fr{m}_2]_{\fr{m}_3}\neq \left\{ {0} \right\},\end{eqnarray} 

then $\lambda_1=\lambda_2=\lambda_3$.

\end{prop}

\begin{proof}
We first note that in view of equation (\ref{atp}), the pairwise $B$-orthogonality of any pair of eigenvectors $X,Y$ of $A$ is equivalent to the pairwise $A$-orthogonality of $X,Y$.  Therefore, since the submodules $\fr{m}_i$ are subsets of eigenspaces of $A$, the $B$-orthogonality and $A$-orthogonality assumptions for the $\fr{m}_i$ in Propositions \ref{triple} and \ref{ni} respectively, are equivalent.\\
For the proof of part \textbf{1.}, suppose that $\lambda_1\neq \lambda_2$.  By taking into account relation (\ref{tr1}) and the fact that $A$ is a $G$-GO metric, then equation (\ref{e1}) of Proposition \ref{ni} implies that $[\fr{m}_1,\fr{m}_2]\subset \fr{m}_1\oplus \fr{m}_2$, which contradicts equation (\ref{tr2}) in the hypothesis.  Therefore, $\lambda_1=\lambda_2$.\\
For the proof of part \textbf{2.}, we first observe that by virtue of part \textbf{1.}, equations (\ref{tr3}) and (\ref{tr4}) imply that $\lambda_1=\lambda_2$.  It remains to prove that $\lambda_1=\lambda_3$.  Since $\lambda_1=\lambda_2$, equation (\ref{tr3}) implies that $\fr{m}_1,\fr{m}_2\subset \fr{m}_{\lambda_1}$.  Equation (\ref{e2}) of Proposition \ref{ni}, in turn implies that 

\begin{equation*}\label{then}[\fr{m}_1,\fr{m}_2]\subset \fr{h}\oplus \fr{m}_{\lambda_1}.\end{equation*}

Therefore, we have that 

\begin{equation}\label{5}[\fr{m}_1,\fr{m}_2]_{\fr{m}_3}\subset \fr{m}_{\lambda_1}.\end{equation}

However, by equation (\ref{tr3}) we have that $\fr{m}_3\subset \fr{m}_{\lambda_3}$ hence

\begin{equation}\label{6}[\fr{m}_1,\fr{m}_2]_{\fr{m}_3}\subset \fr{m}_{\lambda_3}.\end{equation}
   
	Since $[\fr{m}_1,\fr{m}_2]_{\fr{m}_3}\neq \left\{ {0} \right\}$, then by equations (\ref{5}) and (\ref{6}) we obtain that $\lambda_1=\lambda_3$.  
\end{proof}


\subsection{Diagonalization of GO metrics on isotypical summands}\label{dia}

The following results allow us to simplify the form (\ref{matrixform}) for certain isotypical summands $S_k$, under the condition that $A$ is a $G$-GO metric.\\  
  Let $S_0$ be the isotypical summand defined by relation (\ref{S0}).  Then $S_0$ is the Lie algebra of the group $N_G(H)/H$.  We begin by stating the following general result.

\begin{prop}\label{goniso}
Let $(G/H,A)$ be a compact GO space with $H$ connected.  Then $N_G(H)/H$ endowed with the metric $\left.A\right|_{S_0}$ is a GO space.\end{prop}

\begin{proof}
	
	The isotypical summand $S_0$ is the Lie algebra of the group $N_G(H)/H$.  By virtue of Corollary \ref{poritsis}, the endomorphism $\left.A\right|_{S_0}$ is $\operatorname{Ad}(N_G(H))$-equivariant.  Therefore, $A$ restricts to an $N_G(H)$-invariant GO metric in $N_G(H)/H$.  Moreover, since $gH=Hg$, $g\in N_G(H)$, then any $N_G(H)$-invariant metric in $N_G(H)/H$ is also $N_G(H)/H$-invariant.  Therefore, $\left.A\right|_{S_0}$ is an $N_G(H)/H$-invariant metric on $N_G(H)/H$.\\  It remains to show that $\left.A\right|_{S_0}$ is a GO metric.  Indeed, choose $X\in S_0$.  Since $A$ is a GO metric in $G$, there exists an $a_X\in \fr{h}$ such that
	
	\begin{equation}\label{sat1}[a_X+X,AX]=0.\end{equation}
	 
	Taking into account that $X\in S_0$ we obtain that 
	
	\begin{equation}\label{sat2}AX=\left.A\right|_{S_0}X\in S_0.\end{equation}
		
		Moreover, since $S_0$ is the Lie algebra of the group $N_G(H)/H$ we have that 
			
			\begin{equation}\label{sat3}[S_0,X]=\{0\}.\end{equation}
			
			Using equations (\ref{sat2}) and (\ref{sat3}) in equation (\ref{sat1}) we obtain that
	
	\begin{equation*}\label{sat}[X,\left.A\right|_{S_0}X]=0,\end{equation*}
	
\noindent which, by virtue of Proposition \ref{G}, implies that $\left.A\right|_{S_0}X$ is a GO metric.  Therefore $N_G(H)/H$ is endowed with an $N_G(H)/H$-invariant GO metric $\left.A\right|_{S_0}$, hence $(N_G(H)/H,\left.A\right|_{S_0})$ is a GO space.
	\end{proof}

	In order to simplify $\left.A\right|_{S_0}$ we need the following proposition which characterizes the $G$-invariant GO-metrics in any compact Lie group $G$.

\begin{prop}\label{enn}
Let $G$ be a compact Lie group with Lie algebra $\fr{g}$ and let

\begin{equation*}\fr{g}=\fr{z}(\fr{g})\oplus \fr{s}_1\oplus \cdots \oplus \fr{s}_m,\end{equation*}

\noindent be the $B$-orthogonal decomposition of $\fr{g}$ into the direct sum of its center $\fr{z}(\fr{g})$ and simple ideals $\fr{s}_i$, where $B$ is an $\operatorname{Ad}$-invariant inner product on $\fr{g}$.  A $G$-invariant metric $A$ in $G$ is a $G$-GO metric if and only if $A$ has the form

\begin{equation}\label{form}A=C_{\fr{z}(\fr{g})}+\lambda_1\left.\operatorname{Id}\right|_{\fr{s}_1}+\cdots +\lambda_m\left.\operatorname{Id}\right|_{\fr{s}_m},\end{equation}

where $C_{\fr{z}(\fr{g})}:\fr{z}(\fr{g})\rightarrow \fr{z}(\fr{g})$ is any positive definite symmetric operator.
\end{prop}

\begin{proof}
We recall that a $G$-invariant metric $A$, in a compact Lie group $G$, is a $G$-GO metric if and only if $A$ is a bi-invariant metric (\cite{Al-Ni}, Proposition 8).  Moreover, any bi-invariant metric $A$ in $G$ has the form 	
		
		\begin{equation}\label{form}A=C_{\fr{z}(\fr{g})}+\lambda_1\operatorname{Id}|_{\fr{s}_1}+\cdots +\lambda_m\operatorname{Id}|_{\fr{s}_m}\end{equation}

(\cite{D-Z},\cite{Mi}).
This completes the proof.
\end{proof}
 
We obtain the following.

\begin{prop}\label{en}Let $M=G/H$ be a compact homogeneous manifold with $H$ connected.   Assume the reductive decomposition $\fr{g}=\fr{h}\oplus \fr{m}$ with respect to an $\operatorname{Ad}$-invariant inner product $B$ in $\fr{g}$, and let $S_0\subset \fr{m}$ be the isotypical summand defined by relation (\ref{S0}).  Moreover, let $S_0=\fr{z}(S_0)\oplus \fr{s}_1\oplus \cdots \oplus \fr{s}_m$ be the $B$-orthogonal decomposition of $S_0$ as a direct sum of abelian and simple ideals.  Then for any $G$-GO metric $A$ in $M$ it is  

\begin{equation*}\label{form}\left.A\right|_{S_0}=C_{\fr{z}(S_0)}+\lambda_1\operatorname{Id}|_{\fr{s}_1}+\cdots +\lambda_m\operatorname{Id}|_{\fr{s}_m},\end{equation*}

where $C_{\fr{z}(S_0)}:\fr{z}(S_0)\rightarrow \fr{z}(S_0)$ is any positive definite symmetric operator.
\end{prop}

\begin{proof}
    The result is obtained by combining Propositions \ref{goniso} and \ref{enn}.
\end{proof}	
	
In the following proposition we use certain values of the bracket $[S_k^{\bot},S_k]$, where $S_k$ is an isotypical summand, in order to diagonalize any $G$-GO metric $A$ restricted on $S_k$.

\begin{prop}\label{class1}Let $M=G/H$ be a compact homogeneous manifold, with the reductive decomposition $\fr{g}=\fr{h}\oplus \fr{m}$ with respect to an $\operatorname{Ad}$-invariant inner product $B$ in $\fr{g}$.  Choose a $B$-orthogonal basis $\mathcal{B}$ of $\fr{m}$ and let $S_k=\fr{m}_{k_1}\oplus \cdots \oplus \fr{m}_{k_r}\subset \fr{m}$ be an isotypical summand of the isotropy representation of $H$ in $\fr{m}$.  Assume that for each submodule $\fr{m}_{k_l}$, $l=1,\dots,r$ and for every $v\in S_k$ with non-zero projection on $\fr{m}_{k_l}$, there exists a vector $X_l^v\in S_k^{\bot}$ such that
\begin{equation*}\label{ssst}[X_l^v,v]\in \fr{m}_{k_l}\setminus \left\{ {0} \right\}.\end{equation*}
Then any $G$-GO metric $A$ in $M$ is diagonal on $S_k$, i.e. 

\begin{equation*}\label{inis}\left.A\right|_{S_k}=\mu_1\left.\operatorname{Id}\right|_{\fr{m}_{k_1}}+ \cdots + \mu_r\left.\operatorname{Id}\right|_{\fr{m}_{k_r}}.\end{equation*}
\end{prop}

\begin{proof}
We recall the general form (\ref{matrixform}) of the matrix $\left.A^{\mathcal{B}}\right|_{S_k}$.  In order to prove Proposition \ref{class1}, it suffices to prove that the block matrices $A_{lm}$ in (\ref{matrixform}) are zero, for any $l,m=1,\dots,r$ with $l\neq m$.\\
Choose any $l_o=1,\dots,r$.  Since the restricted endmorphism $\left.A\right|_{S_k}:S_k\rightarrow S_k$ is symmetric, then $S_k$ admits a basis consisting of eigenvectors of $\left.A\right|_{S_k}$.  Therefore, there exists an eigenvector $v_{l_o}$ of $\left.A\right|_{S_k}$, such that $v_{l_o}$ has non-zero projection on the submodule $\fr{m}_{k_{l_o}}\subset S_k$.  Let $\lambda$ be the corresponding eigenvalue of $v_{l_o}$, i.e. $v_{l_o}\in \fr{m}_{\lambda}$.  Moreover, consider the vector $X_{l_o}^{v_{l_o}}\in S_k^{\bot}$ as in the hypothesis of Proposition \ref{class1}.  We set 
\begin{equation*}Z_o=[X_{l_o}^{v_{l_o}},v_{l_o}]\in \fr{m}_{k_{l_o}}\setminus \left\{ {0} \right\}.\end{equation*}
We will first prove that $Z_o$ is also an eigenvector of $\left.A\right|_{S_k}$ with eigenvalue $\lambda$.\\
Indeed, let $\mu\neq \lambda$ be another eigenvalue of $\left.A\right|_{S_k}$ and let $Y\in \fr{m}_{\mu}\subset S_k$ be a corresponding eigenvector.  Since $A$ is a $G$-GO metric, Proposition \ref{ni} implies that 
\begin{equation}\label{impth}[v_{l_o},Y]\in [\fr{m}_{\lambda},\fr{m}_{\mu}]\subset \fr{m}_{\lambda}\oplus \fr{m}_{\mu}\subset S_k.\end{equation}
Moreover, taking into account the fact that $X_{l_o}^{v_{l_o}}\in S_k^{\bot}$ as well as relation (\ref{impth}), and using the $\operatorname{ad}$-skew symmetry of $B$, we obtain that
\begin{equation*}B(Z_o,Y)=B([X_{l_o}^{v_{l_o}},v_{l_o}],Y)=B(X_{l_o}^{v_{l_o}},[v_{l_o},Y])=0,\end{equation*}
which immediately implies that the vector $Z_o$ is orthogonal to $\fr{m}_{\mu}$.  Since the eigenvalue $\mu$ is arbitrary, we deduce that $Z_o$ is orthogonal to any eigenspace $\fr{m}_{\mu}$, $\mu\neq \lambda$.  By also taking into account that the summand $S_k$ admits a $B$-orthogonal decomposition into eigenspaces of $\left.A\right|_{S_k}$, we conclude that $Z_o$ is an eigenvector of $\left.A\right|_{S_k}$ with eigenvalue $\lambda$, i.e. 
\begin{equation}\label{d2}\left.A\right|_{S_k}Z_o=\lambda Z_o.\end{equation}\\
On the other hand, $Z_o\in \fr{m}_{k_{l_o}}$, therefore by using equation (\ref{mfe}) we have that 
\begin{equation}\label{d1}\left.A\right|_{S_k}Z_o=\mu_{l_o}Z_o+\sum_{l_o\neq m=1}^rA_{l_om}Z_o,\end{equation}
where $A_{l_om}Z_o\in \fr{m}_{k_m}$, $l_o\neq m=1,\dots,r$.  Equations (\ref{d2}) and (\ref{d1}) imply that $\lambda=\mu_{l_o}$ and also that 
\begin{equation}\label{mpsm}A_{l_om}Z_o=0,\quad m=1,\dots,r,\quad m\neq l_o.\end{equation}
However, the matrices $A_{lm}$ correspond to $\operatorname{Ad}(H)$-equivariant maps $\phi:\fr{m}_{k_l}\rightarrow \fr{m}_{k_m}$.  Since $Z_o\neq 0$, and since any non-zero equivariant map between irreducible submodules is an isomorphism, equation (\ref{mpsm}) implies that $A_{l_om}=0$ for any $m=1,\dots r$ with $m\neq l_o$.  Finally, since $l_o$ is arbitrary we obtain that $A_{lm}=0$ for any $l,m=1,\dots,r$, $l\neq m$, which concludes the proof of Proposition \ref{class1}.
\end{proof}

The following proposition implies that the restriction of any $G$-GO metric on an isotypic summand $S_k$ is a scalar multiple of the identity if the non-zero $\operatorname{Ad}(H)$-equivariant maps between the submodules in $S_k$ satisfy certain orthogonality conditions.

\begin{prop}\label{class}Let $M=G/H$ be a compact homogeneous manifold with the reductive decomposition $\fr{g}=\fr{h}\oplus \fr{m}$ with respect to an $\operatorname{Ad}$-invariant inner product $B$ in $\fr{g}$.  Choose a $B$-orthogonal basis $\mathcal{B}$ of $\fr{m}$ and let $S_k=\fr{m}_{k_1}\oplus \cdots \oplus \fr{m}_{k_r}\subset \fr{m}$ be an isotypical summand of the isotropy representation of $H$ in $\fr{m}$.  Assume that for each $l=1,\dots,r$ there exists a vector $X_l\in \fr{m}_{k_l}$ such that for any pair $m_1,m_2=1,\dots,r$, with $l,m_1,m_2$ distinct, and for any pair of non-zero $\operatorname{Ad}(H)$-equivariant maps $\phi_{lm_1}:\fr{m}_{k_l}\rightarrow \fr{m}_{k_{m_1}}$, $\phi_{lm_2}:\fr{m}_{k_l}\rightarrow \fr{m}_{k_{m_2}}$, the following relations hold.

\begin{equation}\label{itis} [X_l,\phi_{lm_i}(X_l)]_{S_k^{\bot}}\neq 0, \quad i=1,2,\end{equation}
 and 
\begin{equation}\label{itis1} B([X_l,\phi_{lm_1}(X_l)]_{S_k^{\bot}},[X_l,\phi_{lm_2}(X_l)]_{S_k^{\bot}})=0.\end{equation}

Then any $G$-GO metric $A$ in $M$ is scalar on $S_k$, i.e. 
\begin{equation*}\left.A\right|_{S_k}=\lambda \left.\operatorname{Id}\right|_{S_k}. \end{equation*}

\end{prop}

\begin{proof}
At first, we prove that $A$ is diagonal on $S_k$, i.e. 

\begin{equation}\label{ts1}\left.A\right|_{S_k}=\mu_1\left.\operatorname{Id}\right|_{\fr{m}_{k_1}}+\cdots+\mu_r\left.\operatorname{Id}\right|_{\fr{m}_{k_r}} , \quad \mu_1,\dots,\mu_r>0.\end{equation}
 In view of equation (\ref{matrixform}), it is equivalent to prove that the $\operatorname{Ad}(H)$-equivariant maps $A_{lm}$ are zero for any $l,m=1,\dots,r$ with $l\neq m$.\\
Choose any $l_o=1,\dots,r$, and let $X_{l_o}\in \fr{m}_{k_{l_o}}$ be as in the hypothesis of Proposition \ref{class}.  Assume that there exists an $m_o=1,\dots,r$, with $m_o\neq l_o$, such that $A_{l_om_o}$ is non-zero.  We will prove that the above assumption leads to a contradiction.\\
From relation (\ref{mfe}), we obtain that 

\begin{equation}\label{d}AX_{l_o}=\left.A\right|_{S_k}X_{l_o}=\mu_{l_o}X_{l_o}+\sum_{l_o\neq m=1}^rA_{l_om}X_{l_o}.\end{equation}

Moreover, since $A$ is a $G$-GO metric, then Proposition \ref{G} implies that there exists a vector $a\in \fr{h}$ such that 

\begin{equation}\label{s}[a+X_{l_o},AX_{l_o}]=0.\end{equation}

By substituting expression (\ref{d}) into equation (\ref{s}) we obtain that 

\begin{equation}\label{ff}\mu_{l_o}[a,X_{l_o}]+\sum_{l_o\neq m=1}^r[a,A_{l_om}X_{l_o}]+\sum_{l_o\neq m=1}^r[X_{l_o},A_{l_om}X_{l_o}]=0.\end{equation}

We will show that the assumption that $A_{l_om_o}$ is non-zero, contradicts equation (\ref{ff}).\\
 Since $X_{l_o}$ is a vector in $\fr{m}_{k_{l_o}}$, and since each of the submodules $\fr{m}_{k_{l_o}},\fr{m}_{k_m}$, $l_o\neq m=1,\dots,r$, are $\operatorname{ad}(\fr{h})$-invariant, we have that 

\begin{equation*}[a,X_{l_o}]\in \fr{m}_{k_{l_o}}\quad \makebox{and} \quad [a,A_{l_om}X_{l_o}]\in [a,\fr{m}_{k_m}]\subset \fr{m}_{k_m}.\end{equation*}

  Hence 
\begin{equation}\label{ee1}\mu_{l_o}[a,X_{l_o}]+\sum_{l_o\neq m=1}^r[a,A_{l_om}X_{l_o}]\in S_k.\end{equation}
		
On the other hand, the assumption that $A_{l_om_o}$ is a non-zero $\operatorname{Ad}(H)$-equivariant map, along with equation (\ref{itis}) in the hypothesis, imply that 

\begin{equation}\label{ee2}[X_{l_o},A_{l_om_o}X_{l_o}]_{S_k^{\bot}}\neq 0.\end{equation}

 Moreover, by virtue of equations (\ref{itis1}), the terms $[X_{l_o},A_{l_om}X_{l_o}]_{S_k^{\bot}}$, $l_o\neq m=1,\dots,r$ are pairwise orthogonal in any of the following separate cases:\\
\textbf{i.} Each $A_{l_om}$, $l_o\neq m=1,\dots,r$, is a non-zero $\operatorname{Ad}(H)$-equivariant map.\\
\textbf{ii.} There exist an $m=1,\dots,r$, with $l_o\neq m$, such that $A_{l_om}=0$.\\

Therefore, equation (\ref{ee2}) implies that 

\begin{equation*}\label{ee3}\sum_{l_o\neq m=1}^r[X_{l_o},A_{l_om}X_{l_o}]_{S_k^{\bot}}\neq 0,\end{equation*} 

\noindent which, along with equation (\ref{ee1}), contradict equation (\ref{ff}).  We conclude that $A_{l_om_o}=0$.  Since $l_o,m_o$ are arbitrary we obtain that $A_{lm}=0$ for any $l,m=1,\dots,r$ with $l\neq m$, therefore, $A$ is diagonal on $S_k$.\\
To conclude the proof, it remains to prove that all eigenvalues $\mu_l$, $l=1,\dots,r$ in relation (\ref{ts1}) are equal, which implies that $\left.A\right|_{S_k}=\lambda \left.\operatorname{Id}\right|_{S_k}$.  To this end, we will use Proposition \ref{triple}.  For any $l,m=1,\dots,r$ with $l\neq m$, relation (\ref{itis}) implies that	$[\fr{m}_{k_l},\fr{m}_{k_m}]_{S_k^{\bot}}\neq \left\{ {0} \right\}$, therefore, it is

\begin{equation*}[\fr{m}_{k_l},\fr{m}_{k_m}]_{(\fr{m}_{k_l}\oplus \fr{m}_{k_m})^{\bot}}\neq \left\{ {0} \right\}.\end{equation*}	

By virtue of part \textbf{1.} of Proposition \ref{triple} we obtain that $\mu_l=\mu_m$.  Since the pair $(l,m)$ is arbitrary, we conclude that all eigenvalues $\mu_l$, $l=1,\dots,r$ are equal, therefore, $A$ is scalar on the summand $S_k$ and Proposition \ref{class} follows.

\end{proof}

\section{Proof of Theorem \ref{mainuu}}\label{prooof}

In this section we apply the results of section \ref{s3} in order to prove Theorem \ref{mainuu} and therefore obtain the $U(n)$-GO metrics in the complex Stiefel manifolds $V_k\mathbb C^n$.  As mentioned in section \ref{ps}, in order to determine all GO metrics in a homogeneous manifold $M$, one has to consider all Lie groups $G$ acting smoothly and transitively on $M$.  For most cases of Stiefel manifolds $M$, the classification of smooth, transitive and effective actions on $M$ is obtained in \cite{Hs-Su}.\\    

The Stiefel manifold $V_k\mathbb C^n$ is the set of orthonormal $k$-frames in $\mathbb C^n$.  The group $U(n)$ acts smoothly and transitively on $V_k\mathbb C^n$ so that $V_k\mathbb C^n$ is diffeomorphic to $U(n)/U(n-k)$, where $U(n-k)$ is embedded diagonally in $U(n)$.  We consider the $\operatorname{Ad}(U(n))$-invariant inner product $B:\fr{u}(n)\times \fr{u}(n)\rightarrow \mathbb R$ given by

\begin{equation*}\label{inner}B(X,Y)=-\operatorname{Trace}(XY), \quad X,Y\in \fr{u}(n)\end{equation*}

and we obtain a $B$-orthogonal reductive decomposition 
\begin{equation}\label{obt}\fr{u}(n)=\fr{u}(n-k)\oplus \fr{m}.\end{equation}

We consider a basis of $\fr{u}(n)$ as follows:\\
Let $M_{n\times n}\mathbb C$ be the set of complex valued $n\times n$ matrices and let $E_{ij}\in M_{n\times n}\mathbb C$, $i,j=1,\dots,n$, be the matrix with value equal to 1 in the $(i,j)$-entry and zero elsewhere.  For $i,j=1,\dots,n$, we set

\begin{equation}\label{mel}e_{ij}=E_{ij}-E_{ji},\quad \bar{e}_{ij}=\sqrt{-1}(E_{ij}+E_{ji}).\end{equation} 
 
Then, the set 

\begin{equation*}\label{set}\mathcal{B}=\left\{ {e_{ij},\bar{e}_{lm}:i<j=1,\dots,n \quad \makebox{and} \quad l\leq m=1,\dots,n} \right\},\end{equation*}  

\noindent constitutes a basis of $\fr{u}(n)$ which is orthogonal with respect to $B$.  The Lie bracket relations between the basis vectors can be obtained by the following lemma.

\begin{lemma}\label{rel}

Regarding the vectors (\ref{mel}), the following relations hold.
\begin{equation*}
\begin{array}{ccc}
e_{ij}=-e_{ji},\quad \bar{e}_{ij}=\bar{e}_{ji}, & [e_{ij},e_{lm}]=\delta_{jl}e_{im}-\delta_{im}e_{lj}-\delta_{il}e_{jm}-\delta_{jm}e_{il},\\

[\bar{e}_{ij},e_{lm}]=\delta_{jl}\bar{e}_{im}-\delta_{im}\bar{e}_{lj}+\delta_{il}\bar{e}_{jm}-\delta_{jm}\bar{e}_{il}, & [\bar{e}_{ij},\bar{e}_{lm}]=-\delta_{jl}e_{im}+\delta_{im}e_{lj}-\delta_{il}e_{jm}-\delta_{jm}e_{il}.\end{array}
\end{equation*}

\end{lemma}

\begin{proof}We observe that $[E_{ij},E_{lm}]=\delta_{jl}E_{im}-\delta_{im}E_{lj}$.  Lemma \ref{rel} then follows by direct computation.
\end{proof}

In the proof of Theorem \ref{mainuu} we will also use the following.

\begin{prop}(\cite{Go-On-Vi},\cite{Ar})\label{ar} Let $M=G/H$ be a homogeneous manifold with the reductive decomposition $\fr{g}=\fr{h}\oplus \fr{m}$.  Let $\chi$ be the isotropy representation of $H$ in $\fr{m}$ and let $h\in H$, $a\in \fr{h}$ and $X\in \fr{m}$.  Then

\begin{equation}\label{split}\operatorname{Ad}^G(h)(a+X)=\operatorname{Ad}^H(h)a+\chi(h)X;\end{equation}

\noindent that is, the restriction $\operatorname{Ad}^G|_H$ splits into the sum $\operatorname{Ad}^H\oplus \chi$.
\end{prop}  

We proceed to prove Theorem \ref{mainuu}.  The proof comprises of three basic steps which appear as the subsections (\ref{q1})-(\ref{q3}).\\

\emph{Proof of Theorem \ref{mainuu}.}

\subsection{Isotropy representation of $V_k\mathbb C^n=U(n)/U(n-k)$}\label{q1}

We determine the irreducible submodules of the isotropy representation $\chi$ of $U(n-k)$ in $\fr{m}$, in terms of the basis $\mathcal{B}$.  We denote by $\mu_n$ the standard representation of $U(n)$ in $\mathbb C^n$.  Then

\begin{equation}\label{st}\operatorname{Ad}^{U(n)}\otimes \mathbb C=\mu_n \otimes_{\mathbb C} \bar{\mu}_n\end{equation}

(\cite{Ar}).  By taking into account the assumption that the subgroup $U(n-k)$ is embedded diagonally in $U(n)$, we can identify $U(n-k)$ with the subgroup

\begin{equation*}
\left\lfloor\begin{array}{ll}
\operatorname{Id}_k & \hspace{2cm} 0 \\
0  & \hspace{25pt} U(n-k)
\end{array}\right\rfloor
\end{equation*}

of $U(n)$.  Thus, by using relation (\ref{st}) we obtain that 

\begin{eqnarray*}\left.\operatorname{Ad}^{U(n)}\otimes \mathbb C\right|_{U(n-k)}&=&\left.\mu_n \otimes_{\mathbb C} \bar{\mu}_n \right|_{U(n-k)}=\left.\mu_n\right|_{U(n-k)}\otimes_{\mathbb C} \left.\bar{\mu}_n\right|_{U(n-k)} \nonumber \\
&=&
(\underbrace{1 \oplus \cdots \oplus 1}_{k-\makebox{times}}\oplus \mu_{n-k})\otimes_{\mathbb C}(\underbrace{1 \oplus \cdots \oplus 1}_{k-\makebox{times}}\oplus \bar{\mu}_{n-k})\ \\\\
&=&
(\mu_{n-k}\otimes_{\mathbb C}\bar{\mu}_{n-k})\oplus (\underbrace{(\mu_{n-k} \oplus \bar{\mu}_{n-k})\oplus \cdots \oplus (\mu_{n-k}\oplus \bar{\mu}_{n-k})}_{k-\makebox{times}})\oplus (\underbrace{1\oplus \cdots \oplus 1}_{k^2-\makebox{times}}).\label{rep} 
\end{eqnarray*}

The summand $\mu_{n-k}\otimes_{\mathbb C}\bar{\mu}_{n-k}$ corresponds to $\operatorname{Ad}^{U(n-k)}\otimes \mathbb C$, therefore, by virtue of Proposition (\ref{ar}), 
$\chi\otimes \mathbb C$ is given by 

\begin{equation*}\label{chi}\chi\otimes \mathbb C=(\underbrace{(\mu_{n-k} \oplus \bar{\mu}_{n-k})\oplus \cdots \oplus (\mu_{n-k}\oplus \bar{\mu}_{n-k})}_{k-\makebox{times}})\oplus (\underbrace{1\oplus \cdots \oplus 1}_{k^2-\makebox{times}}).\end{equation*}

Therefore, the real representation $\chi$ induces a decomposition of $\fr{m}$ into the sum of two isotypical summands, as

\begin{equation}\label{mis}\fr{m}=S_0\oplus S_1,\end{equation}

where $S_0$ is generated by the trivial submodules of $\chi$ and $S_1$ is the sum of $k$ equivalent irreducible submodules, each one corresponding to the term $\mu_{n-k}$ of $\chi$, and having real dimension $2(n-k)$.

In terms of the basis $\mathcal{B}$ we have that 

\begin{equation*}\label{h}\fr{u}(n-k)=\operatorname{span}_{\mathbb R}\left\{ {e_{ij},\bar{e}_{lm}\in \mathcal{B}:k+1\leq i,j,l,m\leq n} \right\},\end{equation*}

and

\begin{equation}\label{S00}S_0=\operatorname{span}_{\mathbb R}\left\{ {e_{ij},\bar{e}_{lm}\in \mathcal{B}:1\leq i,j,l,m\leq k} \right\}.\end{equation}

Moreover, each of the $k$ equivalent irreducible representations $\mu_{n-k}$ corresponds to the $2(n-k)$-dimensional subspace

\begin{equation}\label{mi}\fr{m}_i=\operatorname{span}_{\mathbb R}\left\{ {e_{ij},\bar{e}_{im}\in \mathcal{B}:k+1\leq j,m \leq n} \right\},\quad i=1,\dots,k,\end{equation}
 
so that 
\begin{equation}\label{soth}S_1=\fr{m}_1\oplus \cdots \oplus \fr{m}_k.\end{equation}

  The following is a depiction of the upper triangular component of $\fr{m}$ in $\fr{u}(n)$.  For any submodule $\fr{m}_i$, the corresponding upper triangular component will be denoted by $[\fr{m}_i^u]$.

\begin{eqnarray*}
\label{fig}\left\lfloor \begin{array}{cccccc}

\left\lfloor \begin{array}{ccc}
1& \cdots &1\\
 &  \ddots & \vdots \\
 &  & 1
\end{array}
\right\rfloor_{k\times k}
&
\begin{array}{c}
\left\lfloor\hspace{23pt}\fr{m}_1^u\hspace{16pt}\right\rfloor_{1\times (n-k)}\\
\vdots_{\hspace{20pt}}\\
\left\lfloor\hspace{25pt}\fr{m}_k^u\hspace{15pt}\right\rfloor_{1\times (n-k)}
\end{array}
\\

\begin{array}{c}
\\
\\
\\
\end{array}

&\hspace{21pt}\left\lfloor\begin{array}{c}
\\
\fr{u}(n-k) \\
\\
\end{array}
\right\rfloor_{(n-k)\times (n-k)}

\end{array}
\right\rfloor
\end{eqnarray*}

\subsection{Simplification of the $U(n)$-GO metrics in $V_k\mathbb C^n$}\label{q2}

 Taking into account the decomposition (\ref{mis}) of the tangent space $\fr{m}$, then by virtue of Proposition \ref{summar} we deduce that any $U(n)$-invariant metric $A$ in $V_k\mathbb C^n$ splits as

\begin{equation}\label{mis1}A=\left.A\right|_{S_0}+\left.A\right|_{S_1}.\end{equation}

We assume that $A$ is a $U(n)$-GO metric.  We will use the results of section \ref{s3} in order to simplify the terms $\left.A\right|_{S_0}$ and $\left.A\right|_{S_0}$.\\
At first, we observe from relation (\ref{S00}) that the summand $S_0$ is isomorphic to the Lie algebra $\fr{u}(k)$.  Hence, the decomposition of $S_0$ into the direct sum of its center and simple ideals is given by $S_0=\fr{u}(1)\oplus \fr{su}(k)=\fr{z}(S_0)\oplus \fr{su}(k)$, where $\fr{u}(1)=\fr{z}(S_0)=\operatorname{span}_{\mathbb R}\left\{ {\sum_{i=1}^k{\bar{e}_{ii}}} \right\}$, in terms of the basis $\mathcal{B}$, and $\fr{su}(k)=\left\{ {X\in S_0:\operatorname{Trace}(X)=0} \right\}$.  By virtue of Proposition \ref{en}, the metric $\left.A\right|_{S_0}$ is then given by

\begin{equation}\label{s1}\left.A\right|_{S_0}=\mu\left.\operatorname{Id}\right|_{\fr{z}(S_0)}+\lambda\left.\operatorname{Id}\right|_{\fr{su}(k)}.\end{equation}

In view of the decomposition (\ref{soth}) for the isotypical summand $S_1$, and taking into account relation (\ref{mi}), we obtain that any vector $v\in S_1$ is written as

\begin{equation*}v=\sum_{i=1}^k{\sum_{j=k+1}^n{(a_{ij}e_{ij}+b_{ij}\bar{e}_{ij})}},\quad a_{ij},b_{ij}\in \mathbb R,\end{equation*}

where for any $i=1,\dots,k$, the term $v_i=\sum_{j=k+1}^n{(a_{ij}e_{ij}+b_{ij}\bar{e}_{ij})}$ is the projection of $v$ on $\fr{m}_i$.  For each $i=1,\dots,k$ we choose the vector $\bar{e}_{ii}\in S_0\subset S_1^{\bot}$.  Using Lemma \ref{rel} it is straightforward to verify that 

\begin{equation*}[\bar{e}_{ii},v]=2\sum_{j=k+1}^n{(a_{ij}\bar{e}_{ij}-b_{ij}e_{ij})},\end{equation*}

which is a non-zero vector in $\fr{m}_i$ if $v_i$ is non-zero.  Proposition \ref{class1} then implies that $\left.A\right|_{S_1}$ is diagonal, i.e.

\begin{equation}\label{inis}\left.A\right|_{S_1}=\mu_1\left.\operatorname{Id}\right|_{\fr{m}_1}+ \cdots +\mu_k\left.\operatorname{Id}\right|_{\fr{m}_k}.\end{equation}

We will show that the all the eigenvalues $\mu_i$, $i=1,\dots,k$, are equal.  Indeed, for any $i,j=1,\dots,k$ with $i\neq j$, Lemma \ref{rel} implies that 
\begin{equation}\label{temm}[e_{ik+1},e_{jk+1}]=-e_{ij}\in S_0.\end{equation}

  Since $e_{ik+1}\in \fr{m}_i$ and $e_{jk+1}\in \fr{m}_j$, then equation (\ref{temm}) implies that 

\begin{equation}\label{inis1}[\fr{m}_i,\fr{m}_j]_{(\fr{m}_i\oplus \fr{m}_j)^{\bot}}\neq \{0\}.\end{equation}

By relation (\ref{inis}) we immediately obtain that $\fr{m}_i\subset \fr{m}_{\mu_i}$ and $\fr{m}_j\subset \fr{m}_{\mu_j}$.  By also taking into account relation (\ref{inis1}) as well as the fact that $A$ is GO, then the first part of Proposition \ref{triple} implies that $\mu_i=\mu_j$.  Since $i,j$ are arbitrary, we obtain that all the eigenvalues $\mu_i$ are equal, therefore, equation (\ref{inis}) yields

\begin{equation}\label{s2}\left.A\right|_{S_1}=\widetilde{\lambda}\left.\operatorname{Id}\right|_{S_1}.\end{equation} 

Taking into account equations (\ref{mis1}), (\ref{s1}) and (\ref{s2}) we obtain that 

\begin{equation}\label{nnm}A=\mu\left.\operatorname{Id}\right|_{\fr{z}(S_0)}+\lambda\left.\operatorname{Id}\right|_{\fr{su}(k)}+\widetilde{\lambda}\left.\operatorname{Id}\right|_{S_1}.\end{equation}
  We will show that $\lambda=\widetilde{\lambda}$.  Indeed, choosing the vectors $e_{12}\in \fr{su}(k)\subset S_0$ and $e_{1k+1}\in \fr{m}_1\subset S_1$ and using Lemma \ref{rel} again, we obtain that $[e_{12},e_{1k+1}]=-e_{2k+1}\in \fr{m}_2$.  Therefore we obtain that $[S_0,\fr{m}_1]_{(S_0\oplus \fr{m}_1)^{\bot}}\neq \{0\}$.  Taking into account that $S_0\subset \fr{m}_{\lambda}$ and $\fr{m}_1\subset \fr{m}_{\widetilde{\lambda}}$ then the first part of Proposition \ref{triple} implies that $\lambda=\widetilde{\lambda}$.  \\
	Equation (\ref{nnm}) then yields that $A=\mu\left.\operatorname{Id}\right|_{\fr{z}(S_0)}+\lambda\left.\operatorname{Id}\right|_{\fr{su}(k)\oplus S_1}$.  After normalizing $A$, we obtain that $A$ is homothetic to the one-parameter family 
	
	\begin{equation}\label{opf}A_t=\left.\operatorname{Id}\right|_{\fr{su}(k)\oplus  S_1}+t\left.\operatorname{Id}\right|_{\fr{z}(S_0)},\quad t>0.\end{equation}

For $t=1$, $A_1$ is a normal metric in $V_k\mathbb C^n$ with respect to the decomposition (\ref{obt}).\\
In summary, we have shown that if $A$ is a $U(n)$-GO metric in the Stiefel manifold $V_k\mathbb C^n$, then $A$ is homothetic to the family (\ref{opf}).

\subsection{Verification of the candidate $U(n)$-GO metrics in $V_k\mathbb C^n$}\label{q3}

  To conclude the proof of Theorem \ref{mainuu}, it remains to show that the candidate GO-metrics (\ref{opf}) are indeed $U(n)$-GO metrics in $V_k\mathbb C^n$.\\
Let $X\in \fr{m}=S_0\oplus S_1=\fr{su}(k)\oplus S_1\oplus \fr{z}(S_0)$.  By virtue of Proposition \ref{G}, we need to find a vector $a\in \fr{u}(n-k)$ such that 

\begin{equation}\label{ggo}[a+X,A_tX]=0,\quad t\in \mathbb R.\end{equation}

Let $X_{\fr{su}(k)}$, $X_{S_1}$ and $X_{\fr{z}(S_0)}$ be the projections of $X$ on $\fr{su}(k)$, $S_1$ and $\fr{z}(S_0)$ respectively so that

\begin{equation}\label{prj}X=X_{\fr{su}(k)}+X_{S_1}+X_{\fr{z}(S_0)}.\end{equation}
 	
Then in terms of the basis $\mathcal{B}$ we have that

\begin{equation*}\label{a}X_{\fr{z}(S_0)}=r\sum_{i=1}^k{\bar{e}_{ii}},\quad r\in \mathbb{R},\quad \makebox{and} \quad X_{S_1}=\sum_{i=1}^k{\sum_{j=k+1}^n{(a_{ij}e_{ij}+b_{ij}\bar{e}_{ij})}}, \quad a_{ij},b_{ij}\in \mathbb R.\end{equation*}

Moreover, we set

\begin{equation*}\widetilde{X}_{S_1}=\sum_{i=1}^k{\sum_{j=k+1}^n{(b_{ij}e_{ij}-a_{ij}\bar{e}_{ij})}},\end{equation*} 

and for any $t\in \mathbb R$ we choose the vector

\begin{equation*}a_t=r(1-t)\sum_{i=k+1}^n{\bar{e}_{ii}} \in \fr{u}(n-k).\end{equation*}

By using Lemma \ref{rel}, we obtain the following equations.

\begin{equation}\label{sf0} [X_{\fr{z}(S_0)},X_{S_1}]=-2r\widetilde{X}_{S_1}, \quad \makebox{and}\end{equation}

\begin{equation}\label{sf}[a_t,X_{S_1}]=2r(1-t)\widetilde{X}_{S_1}.\end{equation}

Since $S_0=\fr{z}(S_0)\oplus \fr{su}(k)$ is the Lie algebra of the group $N_G(H)/H$ with $G=U(n),H=U(n-k)$, then we obtain that $[\fr{u}(n-k),S_0]=\left\{ {0} \right\}$.  Finally, we recall that $[\fr{z}(S_0),S_0]=\left\{ {0} \right\}$.  Hence the equation  

\begin{equation}\label{sf1}[a_t,X_{\fr{z}(S_0)}]=[a_t,X_{\fr{su}(k)}]=[X_{\fr{z}(S_0)},X_{\fr{su}(k)}]=0.\end{equation}

We will now verify condition (\ref{ggo}) for the above choice of $a=a_t$.\\
By taking into account relation (\ref{prj}), the definition (\ref{opf}) of the metric $A_t$, as well as equations (\ref{sf0}),(\ref{sf}) and (\ref{sf1}), then for any $t\in \mathbb R$ we obtain that 

\begin{eqnarray*}[a_t+X,A_tX]&=&[a_t+X_{\fr{su}(k)}+X_{S_1}+X_{\fr{z}(S_0)}, X_{\fr{su}(k)}+X_{S_1}+tX_{\fr{z}(S_0)}]\\
&=&
[a_t,X_{\fr{su}(k)}+X_{S_1}+tX_{\fr{z}(S_0)}]+(1-t)[X_{\fr{z}(S_0)},X_{\fr{su}(k)}+X_{S_1}]\\
&=&
[a_t,X_{S_1}]+(1-t)[X_{\fr{z}(S_0)},X_{S_1}]=2r(1-t)\widetilde{X}_{S_1}-2r(1-t)\widetilde{X}_{S_1}=0,\end{eqnarray*}

\noindent which concludes the proof of Theorem \ref{mainuu}.

\hfill\(\Box\)\\

In the particular case where $k=n-1$, i.e. $V_k\mathbb C^n$ is diffeomorphic to $U(n)/U(n-1)$, the GO metrics (\ref{opf}) correspond to deformations of the normal metric $A_1$, along the fibers of the \emph{Hopf fibration}

\begin{equation*}U(1)\rightarrow U(n)/U(n-1)\rightarrow U(n)/U(n-1)\times U(1).\end{equation*}

Therefore, the spaces $(V_{n-1}\mathbb C^n,A_t)$ are precisely the weakly symmetric Berger spheres (\cite{Berger}).

\begin{remark}It was brought to my attention by Professor Y. G. Nikonorov that the proof of Theorem \ref{mainuu} could be simplified by using Proposition \ref{nik}, which is stated as a Corollary in \cite{Ni-2}.  Indeed the endomorphism $\left.A\right|_{S_1}$ in the proof of Theorem \ref{mainuu} can be diagonalized accordingly.  We present an alternative proof of the fact that $\left.A\right|_{S_1}=\left.\widetilde{\lambda} \operatorname{Id}\right|_{S_1}$, that makes use of Proposition \ref{nik}.\\
\end{remark}

\emph{Alternative proof of $\left.A\right|_{S_1}=\left.\widetilde{\lambda} \operatorname{Id}\right|_{S_1}$}.\\

 We consider the homogeneous fibration $\pi:G/H\rightarrow G/N_G(H)$ with fiber $N_G(H)/H$ where $G=U(n)$ and $H=U(n-k)$.  By using the diagonal embedding of $H$ in $G$ we deduce that $N_G(H)=U(k)\times U(n-k)$ hence $G/N_G(H)$ is the complex Grassmannian $U(n)/U(k)\times U(n-k)$.  We can write the tangent space $\fr{m}=T_p(G/H)$ as $\fr{m}=S_0\oplus S_1$, where $S_0=T_p(N_G(H)/H)$ and $S_1=T_{\pi(p)}(G/N_G(H))$ and both $S_0$ and $S_1$ are $A$-invariant.  By Corollary \ref{poritsis} the endomorphism $\left.A\right|_{S_1}$ is $\operatorname{Ad}(N_G(H))$-equivariant, therefore it defines a $N_G(H)$-invariant metric on the Grassmannian $G/N_G(H)$.  However, the Grassmannian is isotropy irreducible, hence $\left.A\right|_{S_1}=\widetilde{\lambda}\left.\operatorname{Id}\right|_{S_1}$.    

\hfill\(\Box\)\\

\end{document}